\documentclass[11pt]{amsart}
\usepackage{comment}

\usepackage[utf8]{inputenc}
\usepackage[T1]{fontenc}
\DeclareMathAlphabet{\mathpzc}{OT1}{pzc}{m}{it}
\usepackage{geometry}
\usepackage[dvipsnames]{xcolor}
\usepackage{xstring}
\newcommand{\mh}[2][e]{%
    \IfStrEqCase{#1}{%
        {e}{\textcolor{Blue}{\textbf{*Mary: #2*}}}%
        {c}{}%
    }
}
\newcommand{\nd}[2][e]{%
    \IfStrEqCase{#1}{%
        {e}{\textcolor{PineGreen}{\textbf{*Nathan: #2*}}}%
        {c}{}%
    }
}
\newcommand{\mhadd}[2][e]{%
    \IfStrEqCase{#1}{%
        {e}{\textcolor{Blue}{\textbf{#2}}}%
        {r}{#2}%
    }
}
\newcommand{\ndadd}[2][e]{%
    \IfStrEqCase{#1}{%
        {e}{\textcolor{PineGreen}{\textbf{#2}}}%
        {r}{#2}%
    }
}
\usepackage{amsfonts}
\usepackage{amssymb}
\usepackage{amsthm}
\usepackage{amsmath}
\usepackage{amscd}
\usepackage[shortlabels]{enumitem}
\usepackage{mathrsfs}
\usepackage{tikz}
\usetikzlibrary{calc,arrows,decorations.pathreplacing}
\usepackage{nicefrac, xfrac}
\usepackage{mathtools,xparse}
\usepackage{soul}
\usepackage{appendix}

\usepackage[pagebackref, colorlinks = true,
linkcolor = red,
urlcolor  = blue,
citecolor = blue,
anchorcolor = blue]{hyperref}

\theoremstyle{plain}

\newtheorem{theorem}{Theorem}[section]
\newtheorem{corollary}[theorem]{Corollary}
\newtheorem{lemma}[theorem]{Lemma}

\newtheorem{proposition}[theorem]{Proposition}

\theoremstyle{definition}
\newtheorem{definition}[theorem]{Definition}

\newtheorem{remark}[theorem]{Remark}

\newtheorem*{theorem*}{Theorem}

\renewcommand{\epsilon}{\varepsilon}

\newcommand{\vertiii}[1]{{\left\vert\kern-0.25ex\left\vert\kern-0.25ex\left\vert #1
		\right\vert\kern-0.25ex\right\vert\kern-0.25ex\right\vert}}

\newcommand{\om}{\ensuremath{\omega}}

\newcommand{\al}{\ensuremath{\alpha}}

\newcommand{\ep}{\ensuremath{\epsilon}}

\DeclareSymbolFont{bbold}{U}{bbold}{m}{n}
\DeclareSymbolFontAlphabet{\mathbbold}{bbold}

\DeclareMathOperator{\diam}{diam}

\DeclareMathOperator{\VD}{VD}
\DeclareMathOperator{\Jac}{Jac}

\newcommand{\cO}{\ensuremath{\mathcal{O}}}
\newcommand{\cP}{\ensuremath{\mathcal{P}}}

\newcommand{\cR}{\ensuremath{\mathcal{R}}}

\newcommand{\cV}{\ensuremath{\mathcal{V}}}

\newcommand{\NN}{\ensuremath{\mathbb N}}

\newcommand{\PP}{\ensuremath{\mathbb P}}

\newcommand{\RR}{\ensuremath{\mathbb R}}

\providecommand{\phantomsection}{}
\AtBeginDocument{\let\textlabel\label}
\makeatletter
\newcommand{\mylabel}[2]{\raisebox{.7\normalbaselineskip}{\phantomsection}(#1)
	\def\@currentlabel{#1}\textlabel{#2}}
\makeatother

\makeatletter
\newcommand\xlabel[2][]{\phantomsection\def\@currentlabelname{#1}\label{#2}}
\makeatother

\ExplSyntaxOn
\NewDocumentCommand{\mathlist}{ O{,} m m }
 {
  \egreg_mathlist:nnn { #1 } { #2 } { #3 }
 }

\seq_new:N \l__egreg_mathlist_seq
\cs_new_protected:Npn \egreg_mathlist:nnn #1 #2 #3
 {
  \seq_set_split:Nnn \l__egreg_mathlist_seq { #1 } { #3 }
  \seq_use:Nnnn \l__egreg_mathlist_seq { #2 } { #2 } { #2 }
 }
\ExplSyntaxOff

\allowdisplaybreaks

\numberwithin{equation}{section}
\title[]{Multifractal Analysis of  equilibrium states of endomorphisms of $\mathbb P^k$}
\date{\today}

\author[N. Dalaklis]{Nathan Dalaklis}
\address[N. Dalaklis]{Department of Mathematics, University of Oklahoma, Norman, OK 73019, USA}
\email{\href{ndalaklis@ou.edu}{ndalaklis@ou.edu}}

\author[Y.M. He]{Yan Mary He}
\address[Y.M. He]{Department of Mathematics, University of Oklahoma, Norman, OK 73019, USA}
\email{\href{he@ou.edu}{he@ou.edu}}

\begin{document}

\begin{abstract}
Let $f$ be a holomorphic endomorphism of $\mathbb C \mathbb P^k$ of algebraic degree at least 2
and let $X \subseteq \mathbb C \mathbb P^k$ be an uniformly expanding set. In this paper, we study multifractal analysis of equilibrium states of H\"older continuous functions for the {\it non-conformal} dynamical system $f \colon X \to X$. In lieu of Hausdorff dimensions, we use a new dimension theory (i.e., the volume dimension theory) 
to define various local dimension multifractal spectra and 
show that each of these spectra form a Legendre transform pair with the temperature function as in the conformal case. As an application of our main theorems, we also prove a conditional variational principle for such dimension multifractal spectra.
\end{abstract}
	
\maketitle

\section{Introduction}

\subsection{Motivation and background}
If $f \colon X \to X$ is a dynamical system, many of its important properties can be characterized using {\it local asymptotic quantities} $\varphi_\infty(x) = \lim_{n \to \infty}\varphi_n(x)$ such as Lyapunov exponents, local dimensions and local entropies. The distribution of these local values across the space $X$ often forms a {\it multifractal structure} of $X$. In particular, one considers level sets $J(\alpha):=\{x \in X: \varphi_\infty(x) = \alpha\}$ and obtains a multifractal decomposition of $X$.
Multifractal analysis provides a global statistical picture of these local asymptotic quantities and interprets the geometry of their variation on $X$. In particular, one defines a {\it multifractal spectrum} $S(\alpha):={\mathcal{S}}(J(\alpha))$ where ${\mathcal{S}}$ is a quantity measuring the size of the set $J(\alpha)$ such as Hausdorff dimension or topological entropy.

There has been substantial amount of work on  multifractal analysis for {\it conformal} dynamical systems. In \cite{Besi35}, Besicovitch studied the Hausdorff dimension of sets determined by the frequency of the
digits in dyadic expansions, which is essentially a multifractal analysis of the Birkhoff averages of the indicator functions for the doubling map. 
Multifractal analysis of the Birkhoff averages of a continuous potential on various dynamical systems such as a mixing subshift of finite type is studied in \cite{BS01,BSS02_nt,BSS02_app,FF00,Hof10,OW07,TV03,Tem01}. Pesin-Weiss \cite{PesinWeiss97} studied a multifractal analysis of equilibrium measures for conformal expanding maps.
Roy-Urbanski \cite{RU09} studied the multifractal analysis of the conformal measure associated to a family of weights imposed upon a graph directed Markov system. We refer the reader to the survey paper \cite{Clim14} for recent development.

On the other hand, the multifractal analysis of {\it non-conformal} dynamical systems is substantially more difficult and remains largely undeveloped. Conformal maps expand or contract equally in all directions, which allows one to obtain distortion lemmas to control how sets are stretched and to compare scales uniformly. However, non-conformal maps allow arbitrarily large distortion, even on small scales. This lack of conformality complicates the geometric structure of invariant sets and measures, causing classical methods to fail and suggesting that new techniques are needed.
In \cite{Barreirabook11}, Barreira-Gelfert used {\it non-additive thermodynamic formalism} to study the multifractal analysis of the level sets of Lyapunov exponents for a class of non-conformal repellers. In \cite{JordanRams11, King95}, the authors studied the local dimension spectra for Bernoulli measures on Bedford-McMullen carpets. More recent work by Cao-Pesin-Zhao \cite{CPZ19} has made progress towards understanding the dimensions of the whole repeller, but they do not preform a multifractal analysis or estimates of the spectra therein. In the case of the plane, Falconer-Fraiser-Lee \cite{Fal2021} and Qiu-Wang-Wang \cite{Q2024} investigated the $L^q$-spectra for planar iterated function systems and planar graph directed systems with restrictions on the contractive behavior of the mappings.

In this paper, we consider the non-conformal multifractal analysis of equilibrium states in the setting of complex dynamics and in particular holomorphic endomorphisms of $\mathbb C\mathbb P^k$, which are generalizations of rational maps on the Riemann sphere to higher dimensional complex projective spaces. The recent work \cite{BHMane} of Bianchi and the second author introduces a delicate dimension theory (i.e., the volume dimension) for expanding measures and sets of such maps. 
Thanks to the volume dimension, in this paper, we study a multifractal analysis of equilibrium states for a holomorphic endomorphism of $\mathbb C\mathbb P^k$ on an uniformly expanding set.

\subsection{Statement of results}
Let $f$ be a holomorphic endomorphism of $\mathbb P^k = \mathbb C\mathbb P^k$ of algebraic degree at least 2.
Let $X \subseteq \mathbb P^k$ be an uniformly expanding set; that is $X$ is a closed invariant set and there exist $\eta > 1$ and $C>0$ such that
$||Df_x^n(v)|| > C\eta^n||v||$ for every $x \in X$, $v \in T_x \mathbb P^k$, and $n\in\mathbb N$.
Consider the dynamical system $f \colon X \to X$. Let $g \colon X \to \mathbb R$ be a H\"older continuous function with pressure $\mathcal{P}(g)=0$. For $(q,t) \in \mathbb R^2$, we consider the following 2-parameter family of potentials $\phi_{q,t} \colon X \to \mathbb R$ given by
\begin{equation*} 
\phi_{q,t} := qg - t\log |\Jac f|.
\end{equation*}
For each fixed $q \in \mathbb R$, the pressure function $t \mapsto \mathcal{P}(\phi_{q,t})$ is strictly decreasing with $\mathcal{P}(\phi_{q,0})>0$; see Lemma \ref{lem_property_qt} \eqref{ITM:SDP}. We define $T(q)$ to be the unique number such that $\mathcal{P}(\phi_{q,t}) = 0$. The function $q \mapsto T(q)$ is called the {\it temperature function}, which is one of the main objects of study in multifractal analysis.

Another main object of study in multifractal analysis is a fine dimension multifractal spectrum. For conformal dynamical systems, it is often the so-called {\it fine Hausdorff dimension multifractal spectrum}. The idea is to decompose the set $X$ as level sets $\{X_\alpha\}_{\alpha\ge0}$ of some {\it local dimension} of a fixed measure $\nu$, together with an irregular set consisting of points in $X$ at which the local dimension does not exist. The local dimension at $x \in X$ is defined as $\lim_{\kappa \to 0}\frac{\log \nu(B(x,\kappa))}{\log \kappa}$, where $B(x,\kappa)$ is the ball of radius $\kappa$ around $x$. Then the fine Hausdorff dimension multifractal spectrum is the function $\alpha \mapsto {\rm HD}(X_\alpha)$, where HD denotes the Hausdorff dimension.

In our setting, however, the dynamical system $f \colon X  \to X$ is {\it non-conformal}. As a consequence, balls and Hausdorff dimension of sets and measures are no longer intrinsic to the dynamics. Therefore the fine Hausdorff dimension multifractal spectrum may no longer be useful and an appropriate substitute is needed. In \cite{BHMane}, Bianchi and the second author introduced a {\it local volume dimension} and {\it volume dimension} for sets and measures, which depends on the dynamics of $f$ and incorporates the absence of an analogue of Koebe’s theorem and the non-conformality of holomorphic endomorphisms of $\mathbb P^k, k\ge2$. We use such dimensions in lieu of local dimension and Hausdorff dimension in our multifractal analysis.

Let $\nu$ be the equilibrium state of $g$. For each $q \in \mathbb R$, let $\nu_q$ be the equilibrium state of $\phi_{q,T(q)}$. We define
\begin{equation*} 
\alpha(q) := -\frac{\int_X g d\nu_q}{\int_X \log |\Jac f| d\nu_q} >0.
\end{equation*}
Recall from \cite{BHMane} that $\delta_{\nu,x}$ is the {\it local volume dimension} of $\nu$ at $x\in X$. 
For each $q \in \mathbb R$, we define
$$J_\nu(\alpha(q)):= \{x \in X: \delta_{\nu,x} = \alpha(q)\}.$$
Then we obtain a {\it multifractal decomposition} of $X$, namely,
$$X = \cup_{q}J_\nu(\alpha(q)) \cup J_\nu'$$
where $J_\nu':=\{x \in X : \delta_{\nu,x} \text{ does not exist}. \}$. 

Consider the volume dimensions of the sets $J_\nu(\alpha(q))$. We first recall from \cite{BHMane} the {\it volume dimension} $\VD_{f,\mu}(Y)$ of the set $Y \subset X$ with respect to a measure $\mu$.
We define the {\it parametrized volume dimension multifractal spectrum} $\hat{S}_\nu \colon \mathbb R \to \mathbb R$ by
$$\hat{S}_\nu(\alpha(q)) := \VD_{f,\nu_q}(J_\nu(\alpha(q))).$$
We call this volume dimension multifractal spectrum {\it parametrized} as the volume dimension is taken with respect to the measure $\nu_q$ which varies with $q \in \mathbb R$. We note that if all the Lyapunov exponents of $\nu_q$ are equal, by \cite{BHMane}, we recover the Hausdorff dimension of the set, i.e., $\VD_{f,\nu_q}(J_\nu(\alpha(q))) = \frac{1}{k}{\rm HD}(J_\nu(\alpha(q)))$.

If $Y$ is an uniformly expanding set, recall from \cite{BHMane} (see also Definition \ref{def_vd_expandingset}) that the volume dimension $\VD_f(Y)$ is defined. It turns out that $J_\nu(\alpha(q))$ is $f$-invariant (see Remark \ref{rmk_Jnu_invariant}) 
and thus it is an uniformly expanding set. 
We define the {\it fine volume dimension multifractal spectrum} $S_\nu \colon \mathbb R \to \mathbb R$ by
$$S_\nu(\alpha(q)) := \VD_f(J_\nu(\alpha(q))) =\displaystyle\sup_{\mu \in \mathcal{M}^+_{J_\nu(\alpha(q))}(f)}\VD_{f,\mu}(J_\nu(\alpha(q))).$$

While it seems that our (parametrized) volume dimension multifractal spectrum ($\hat{S}_\nu(\alpha)$ or) $S_\nu(\alpha)$ is only defined for $\alpha$ in the range $[\alpha_1,\alpha_2] \subset (0,\infty)$ (see Remark \ref{rmk_range_alpha(q)}) of the function $q \mapsto \alpha(q)$, we will see in Theorem \ref{thm_main_2} that this is not a constraint as the set $J_\nu(\alpha)$ is empty if $\alpha \notin [\alpha_1,\alpha_2]$.

\smallskip

Let $I$ be an interval and let $h \colon I \to \mathbb R$ be a strictly convex $C^2$-function, i.e., $h''(x)>0$ for all $x \in I$. Recall that the {\it Legendre transform} of $h$ is the differentiable function $g \colon \mathbb R \to \mathbb R$ given by
$$g(p) := {\rm max}_{x \in I}(px - h(x)).$$
Two strictly convex functions $h$ and $g$ form a {\it Legendre transform pair} if and only if $g(\alpha) = h(q) + q\alpha$ where $\alpha(q) = -h'(q)$ and $q = g'(\alpha)$.

Our first two main theorems concern properties of the functions $q \mapsto \hat{S}_{\nu}(\alpha(q))$ (resp. $q \mapsto S_{\nu}(\alpha(q))$) and $q \mapsto T(q)$ and state that the functions $\alpha \mapsto -\hat{S}_\nu(-\alpha)$ (resp. $\alpha \mapsto -S_\nu(-\alpha)$) and $q \mapsto T(q)$ form a Legendre transform pair. In the conformal setting, an analogue of the theorem has been proved by several authors; see for example \cite{Barreirabook11,Pesin97book,PesinWeiss97,PU}.

\begin{theorem} \label{thm_main_1}
Let $k \ge 1$ be an integer.
Let $f$ be a holomorphic endomorphism of $\mathbb P^k$ of algebraic degree at least 2, $X \subseteq P^k$ an uniformly expanding set and $\nu \in \mathcal{M}_X^+(f)$ an equilibrium state of a H\"older continuous function $g \colon X \to \mathbb R$. Then the following statements hold.
\begin{enumerate}
\item The temperature function $T \colon \mathbb R \to \mathbb R$ is real-analytic. We have $T(0) = \VD_f(X)$ and $T(1) = 0$.
The derivative of $T$ satisfies $T'(q) = -\alpha(q) <0$ for all $q \in \mathbb R$. Moreover, $T(q)$ is convex. It is strictly convex if and only if $\nu \neq \nu_0$, where $\nu_0$ is the equilibrium state of the potential function $-\VD_f(X)\log|\Jac f|$.
\item The function $\alpha \mapsto \hat{S}_\nu(\alpha)$ is real-analytic and $\hat{S}_\nu(\alpha(q)) = \VD_f(\nu_q)= T(q) + q\alpha(q)$.
\item If $\nu \neq \nu_0$, then the functions $\alpha \mapsto -\hat{S}_\nu(-\alpha)$ and $q \mapsto T(q)$ are strictly convex and form a {\it Legendre transform pair}.
\item If $\nu = \nu_0$, then $T(q)$ is a straight line and $\alpha(q) = \VD_f(X)$ for all $q\in \mathbb{R}$. 

\end{enumerate}
\end{theorem}

\begin{theorem}\label{thm_main_1.5}
Let $f$ and $\nu$ be as in Theorem \ref{thm_main_1}. Then Theorem \ref{thm_main_1} (2) and (3) hold for $S_\nu(\alpha(q))$.

\end{theorem}

The above theorems imply that if $\alpha \in [\alpha_1,\alpha_2]$, then the set $J_\nu(\alpha)$ is non-empty. Our next theorem gives a converse, namely, if $\alpha \notin [\alpha_1,\alpha_2]$, then
$J_{\nu}(\alpha) = \emptyset$. In other words, the $S_\nu(\alpha)$-spectrum is {\it complete}. Completeness of the fine Hausdorff dimension spectrum has been proved by Schmeling \cite{Sch99}; see also \cite[Theorem 9.2.5]{PU} and \cite[Theorem 21.2]{Pesin97book}.
\begin{theorem} \label{thm_main_2}
Let $f$ and $\nu$ be as in Theorem \ref{thm_main_1}.
Then the $S_\nu(\alpha)$-spectrum (resp. $\hat{S}_\nu(\alpha)$-spectrum) is complete; namely, we have
$$\alpha_1 = \inf_{x \in X} \underline{\delta}_{\nu,x} \text{~ and ~ } \alpha_2 = \sup_{x \in X} \overline{\delta}_{\nu,x}.$$
In particular, $J_{\nu}(\alpha) = \emptyset$ if and only if $\alpha \notin [\alpha_1,\alpha_2]$.
\end{theorem}

\medskip

Let $(\Sigma_A^+,\sigma)$ be the symbolic coding of $f \colon X \to X$. Now we define a {\it symbolic} multifractal decomposition of $X$ and {\it symbolic} (parametrized) volume dimension spectrum. In the spirit of Barreira-Saussol \cite{BS01}, 
for each $q \in \mathbb R$, we define the symbolic level set $\hat{X}_q$ using the ratio of Birkhoff sums as
\begin{equation*}
\hat{X}_{q} :=\left\{\omega \in \Sigma_A^+ :
\lim_{n\to\infty}\frac{\sum_{k=0}^{n-1} g(\pi(\sigma^k(\omega)))}{\sum_{k=0}^{n-1}  \log|\Jac f(\pi(\sigma^k(\omega)))|^{-1}} = \alpha(q) \right\}.
\end{equation*}
We define the {\it symbolic parametrized volume dimension multifractal spectrum} $\hat{\mathcal{F}}_\nu \colon \mathbb R \to \mathbb R$ by
$$\hat{\mathcal{F}}_\nu(\alpha(q)) := \VD_{f,\nu_q}(\pi(\hat{X}_{q})).$$
We observe that $\pi(\hat{X}_{q})$ is an uniformly expanding set. 
We define the {\it symbolic fine volume dimension multifractal spectrum} $\mathcal{F}_\nu \colon \mathbb R \to \mathbb R$ by
$$\mathcal{F}_\nu(\alpha(q)) := \VD(\pi(\hat{X}_{q})) =\displaystyle\sup_{\mu \in \mathcal{M}^+_{\pi(\hat{X}_{q})}(f)}\VD_{f,\mu}(\pi(\hat{X}_{q})).$$

Our next theorem gives the multifractal analysis for $\hat{\mathcal{F}}_\nu(\alpha(q))$ and $\mathcal{F}_\nu(\alpha(q))$. 

\begin{theorem}\label{thm_main_3}
Let $f$ and $\nu$ be as in Theorem \ref{thm_main_1}. Then Theorem \ref{thm_main_1} (2) and (3) hold for both $\hat{\mathcal{F}}_\nu(\alpha(q))$ and $\mathcal{F}_\nu(\alpha(q))$.
\end{theorem}

\smallskip 

As a collorary of the above theorems, we obtain the {\it conditional variational principle} for the dimension multifractal spectra.
\begin{corollary} \label{thm_main_4}
Let $f$ and $\nu$ be as in Theorem \ref{thm_main_1}. Then the conditional variational principle holds; that is, for any $\alpha \in [\alpha_1,\alpha_2]$, we have
\begin{equation}
\hat{S}_{\nu}(\al) = \sup\{\VD_f(\rho)\;|\; \rho\in \mathcal{M}_{J_{\nu}(\alpha(q))}^+(f)\}
\end{equation}

The conditional variational principle also holds for $S_{\nu}(\al)$, $\hat{\mathcal{F}}_{\nu}(\al)$ and $\mathcal{F}_{\nu}(\al)$.
\end{corollary} 

It is expected that the volume dimension theory developed in \cite{BHMane} would work in more general contexts, such as smooth endomorphisms of compact Riemannian manifolds with singularities or invariant measures with positive sum of Lyapunov exponents. Therefore, we expect that our multifractal analysis to work in this generality as well.

\subsection{Strategy of the proofs}
Recall that in the conformal case,
the local (Hausdorff) dimension is defined as
$$
    \delta_{\nu,x} := \lim_{r\to 0}\frac{\log \nu(B(x,r))}{\log r} ,
$$
where $\nu$ is the equilibrium state of a fixed potential $g$. We define the associated Hausdorff dimension spectrum function for this local dimension multifractal $S_{\nu}(\alpha(q)) := {\rm HD}(J_{\nu}(\alpha(q)))$.
In this case, the standard approach to the multifractal analysis of $\nu$ (see for example \cite{PesinWeiss97,Urbanski22}) requires that one compare the measure of the ball $B(x,r)$ and the measure of the projection $\pi([\om_1\dots \om_n])$ of a cylinder set of the symbolic coding. The utility of this comparison is that one may then find a {\it Moran cover}; a cover composed of sets in refinements of the Markov partition that determine a good enough cover of the set for dimension calculations; see for example \cite{PesinWeiss97}. Using tools from thermodynamic formalism, one can then derive the Legendre pair relationship between $S_\nu(\alpha(q))$ and the temperature function $T(q)$; namely, $S_{\nu}(\alpha(q)) = T(q)+q\alpha(q)$.
    
In our case, showing the existence of a Moran cover is unnecessary since for the local volume dimension, one is only concerned with the $U(x,\epsilon,\kappa,N)$ sets. These sets are the images of balls under local inverse branches, exactly the sets whose measure is described by the Gibbs property (see \cite[Definition 13.2.1]{Urbanski22}). For this reason, it would appear that we have trivialized our non-conformal analysis by choosing the local volume dimension. However, there are a few complications. First, the local volume dimension is a triple limit and so there is some accounting required here. Second, as we see in Proposition \ref{prop_lem3}, the result one would obtain from this standard approach only concludes on $\VD_{f,\nu_q}(J_\nu(\alpha(q)))$,
the volume dimension of the set with respect to $f$ and the measure $\nu_q$. Here the measure $\nu_q$ is dependent on the parameter $q$ and so our dimension function is changing with $q$, whereas in the Hausdorff dimension case it remains constant. Therefore, by following the standard approach alone, one does not recover a true multifractal analysis with respect to the volume dimension. 

Our goal, however, is to show something stronger.  
Let $g$ be a fixed potential with equilibrium state $\nu$ on $X$, where $X$ is an expanding set
of a holomorphic endomorphism $f$ of $\mathbb P^k$ algebraic degree at least $2$.
There is a multifractal decomposition of $X$ by level sets $J_{\nu}(\alpha(q))$ of local volume dimension. We show that the $f$-volume dimension spectrum function $S_{\nu}(\alpha(q)) := \VD_f(J_{\nu}(\alpha(q)))$ form a Legendre pair with the temperature function; namely, $S_{\nu}(\alpha(q)) = T(q)+q\alpha(q)$. 

In the introduction to volume dimension (see \cite{BHMane}), it is not shown that the local volume dimension is an $f$-invariant. Since a priori we lack knowledge of this invariance, we begin our analysis with $\hat{X}_{q}$, a level set in the symbolic space of the ratio of two Birkhoff sums. We show that $\pi(\hat{X}_{q}) = J_{\nu}(\alpha(q))$ and $\hat{X}_{q}$ is invariant. With the invariance of $J_{\nu}(\alpha(q))$ established, we then utilize \cite[Proposition 4.22]{BHMane} and the Gibbs property to complete our analogue to the standard local dimension analysis. From here our main tool is \cite[Theorem 5.6]{BHMane} which connects $\VD_f(X)$ to $\VD_f(\mu_X)$ where $\mu_X$ is a volume conformal measure. With this fact we establish that $\nu_q = \mu_{J_{\nu}(\alpha(q))}$ (see Lemma \ref{lem_comp_lvd}) and conclude that $\VD_{f,\nu_q}(J_{\nu}(\alpha(q))) = \VD_f(J_{\nu}(\alpha(q)))$. We also have a volume lemma for volume dimensions from \cite[Theorem 1.1]{BHMane} which allows us to retrieve a conditional variational principal for our local volume dimension spectrum.

\subsection{Organization of the paper} The paper is organized as follows. In Section \ref{sec_vd}, we give a brief account of volume dimensions as introduced in \cite{BHMane}. In Section \ref{sec_pres_temp}, we prove properties about the pressure and the temperature functions.
We prove Theorems \ref{thm_main_1}--\ref{thm_main_4} in Sections \ref{sec_main_1}--\ref{sec_main_4}, respectively. Finally, in Section \ref{sec_multi_decomp}, we study properties of sets in the multifractal decomposition of $X$.

\section{Volume dimensions}\label{sec_vd}
In this section, we give a brief overview of the local volume dimension and volume dimension for sets and measures as introduced in \cite{BHMane}.

\subsection{Holomorphic endomorphisms}
Let $f \colon \mathbb P^k \to \mathbb P^k$ be a holomorphic endomorphism of algebraic degree $d \ge 2$.
Recall that $X \subseteq \mathbb P^k$ is an {\it uniformly expanding} set if it is closed and $f$-invariant, and there exist $\eta > 1$ and $C>0$ such that
$||Df_x^n(v)|| > C\eta^n||v||$ for every $x \in X$, $v \in T_x \mathbb P^k$, and $n\in\mathbb N$.
We denote by $\mathcal{M}_X^+(f)$ the set of $f$-invariant ergodic probability measures supported on $X$. Since $X$ is uniformly expanding, we note that any $\nu \in \mathcal{M}_X^+(f)$ has strictly positive Lyapunov exponents. 

\medskip

For the rest of this section, we fix a holomorphic endomorphism $f \colon \mathbb P^k \to \mathbb P^k$, an uniformly expanding set $X$ and $\nu \in \mathcal{M}_X^+(f)$.

\subsection{Local volume dimension}\label{sec_lvd}
In this section, we recall the definition of the {\it local volume dimension} (which depends on $f$ and $\nu$) introduced by Bianchi and the second author in \cite{BHMane}.

A result by Berteloot-Dupont-Molino (see \cite{BD,BDM})
states that there exists an increasing (as $\epsilon\to 0$) measurable exhaustion $\{Z^\star_\nu (\epsilon)\}_\epsilon$ of a full-measure subset $Z^\star_\nu$ of the space of orbits for $f$ such that
 the preimages of sufficiently small balls along  orbits in $Z^\star_\nu(\epsilon)$ 
 are approximately ellipses, and the contraction rate for their volume
 is essentially given
  by $e^{-nL_\nu}$.

\begin{remark}
Since $X$ is uniformly expanding, we have $Z_\nu^*(\ep) = Z_\nu^*$ for any $\ep>0$. However, we keep the parameter $\ep$ in our analysis for the purpose of future generalization of $X$ to be non-uniformly expanding.
\end{remark}

Denote by $\pi\colon Z^\star_\nu\to \mathbb P^k$
the projection associating to any orbit $\hat z =\{z_n\}_{n\in \mathbb Z}$
its element $z_0$. It follows from further estimates developed in \cite{BHMane}, that for every small $\epsilon>0$, there exist some $r(\epsilon)$ and $n(\epsilon)$ such that, for all $x \in \pi(Z^\star_\nu(\epsilon))$, $0<\kappa<r(\epsilon)$, and $N\geq n(\epsilon)$, the neighbourhood $U=U(x,\epsilon,\kappa,N)$ of $x$
satisfying $$f^N (U) = B(f^N(x),\kappa\, e^{-NM\epsilon})$$ where $e^{M}$ is a bound for the expansion of $f$ and  $f^N|_{U}$ is injective is well-defined. We set
\begin{equation*}
	\delta_{\nu,x} (\epsilon,\kappa, N) := \frac{\log \nu(U(x,\epsilon,\kappa,N))}{\log \sqrt{{\rm Vol}(U(x,\epsilon,\kappa,N))}},
\end{equation*}
where $\rm Vol$ denotes the volume 
with respect to the Fubini-Study metric. 

\begin{remark} \label{rmk_change1}
The above definition of $\delta_{\nu,x}$ is different from the definition of $\delta_{\nu,x}$ in \cite{BHMane} by adding an extra square root in the denominator.
\end{remark}

\begin{definition}[{\cite{BHMane}}]\label{def_local_vd}
\begin{enumerate}
\item The {\it lower local volume dimension} $\underline{\delta}_{\nu,x}$ at $x$ is defined as 
$$\underline{\delta}_{\nu,x} := \liminf_{\epsilon \to 0} \liminf_{\kappa \to 0} \liminf_{N \to \infty}\delta_{\nu,x} (\epsilon,\kappa, N).$$
\item The {\it upper local volume dimension} $\overline{\delta}_{\nu,x}$ at $x$ is defined as 
$$\overline{\delta}_{\nu,x} := \limsup_{\epsilon \to 0} \limsup_{\kappa \to 0} \limsup_{N \to \infty}\delta_{\nu,x} (\epsilon,\kappa, N).$$
\item The {\it local volume dimension} $\delta_{\nu,x}$ at $x$ is defined as 
$$\delta_{\nu,x} := \lim_{\epsilon \to 0} \lim_{\kappa \to 0} \lim_{N \to \infty}\delta_{\nu,x} (\epsilon,\kappa, N)$$
if the limits exist.
\end{enumerate}
\end{definition}

\subsection{Volume dimension of sets}
We continue to use the notation from Section \ref{sec_lvd}. In this section, we recall the definition of volume dimension $\VD_{f,\nu}(X)$ of a set $X \subseteq \pi(Z_\nu^{\star}(\epsilon))$ introduced in \cite{BHMane}.

For every $X \subseteq \pi(Z^\star_\nu)$ and $\epsilon>0$, setting $X^\epsilon := X\cap \pi(Z^\star_\nu(\epsilon))$, we define the quantity
$\VD^\epsilon_{f,\nu} ( X^\epsilon)$ as
\[
\VD^\epsilon_{f,\nu} (X^\epsilon) :=
\sup \left\{ \alpha :  \Lambda^\epsilon_\alpha(X^\epsilon) = \infty\right\}=
\inf\left\{ \alpha : \Lambda^\epsilon_\alpha(X^\epsilon) = 0\right\},
\]
where
\[
\Lambda^\epsilon_\alpha (X^\epsilon)
:=
\lim_{\kappa\to 0}
\lim_{N^\star \to \infty} \inf_{\{U_i\}}
\sum_{i \ge 1} {\rm Vol}(U_i)^{\alpha/2}.
\]
Here the infimum is taken over the covers consisting of sets $U_i$ of the form $U_i=U(x,\ep,\kappa,N_i)$, for some $x\in \pi(Z_\nu(\epsilon))$ and $N_i \ge N^\star$.

\begin{remark}\label{rmk_change2}
The above definition of $\Lambda^\epsilon_\alpha (X^\epsilon)$ is different from the definition of $\Lambda^\epsilon_\alpha (X^\epsilon)$ in \cite{BHMane} by adding an extra square root in ${\rm Vol}(U_i)$.    
\end{remark}

\begin{definition}[{\cite{BHMane}}] \label{Def:VDIII}
For every $X\subseteq \pi(Z_\nu^{\star}(\epsilon))$, the {\it volume dimension with respect to $\nu$} of $X$ is 
    \begin{equation}
            \VD_{f,\nu}(X):= \limsup_{\ep\to 0}\VD_{f,\nu}^{\ep}(X^\epsilon).
    \end{equation}
\end{definition}

\begin{definition}[{\cite{BHMane}}]
The {\it volume dimension of a measure $\nu$} is given by 
\begin{equation*}
\VD_f(\nu):= \inf\{\VD_{f,\nu}(X)\;:\; X\subseteq\pi(Z_{\nu}^{\star}) {\rm ~Borel}, \nu(X)=1\}
\end{equation*}
\end{definition}

\begin{definition}[{\cite{BHMane}}] \label{def_vd_expandingset}
Let $X\subseteq \PP^k$ be uniformly expanding. The {\it volume dimension of $X$} is 
\begin{equation*}
\VD(X) := \sup_{\nu \in \mathcal{M}_X^+(f)} \VD_{f,\nu}(X).
\end{equation*}
\end{definition}

\begin{theorem}[{\cite[Theorem 1.1]{BHMane}}]\label{thm_BHMane_1.1}
Let $k \ge 1$ be an integer.
Let $f \colon \mathbb P^k \to \mathbb P^k$ be a holomorphic endomorphism of algebraic degree $d \ge 2$.
For every $\nu\in \mathcal M^+(f)$ we have 
$$\VD_f(\nu) = \frac{h_{\nu}}{L_{\nu}}$$
where $h_\nu$ denotes the entropy of $f$ with respect to $\nu$ and $L_\nu$ denotes the sum of the Lyapunov exponents of $\nu$.
\end{theorem}

\begin{remark}
The above theorem is different from \cite[Theorem 1.1]{BHMane} by a factor of 2 in the denominator, due to the changes in the definition of (local) volume dimensions; see Remarks \ref{rmk_change1} and \ref{rmk_change2}.
\end{remark}

\subsection{Multifractal decompositions of $X$}
We define
\begin{equation}\label{Def:Jal}
{\rm for ~any~} \alpha \ge 0, J_\nu(\alpha) := \{x \in X : \delta_{x} = \alpha \} \text{ and } J_\nu' := \{x \in X : \delta_{x} {\rm ~does ~not ~exist}. \}.
\end{equation}
Then we have a multifractal decomposition of $X$, i.e., $X = J'_\nu \cup \left(\bigcup_{\alpha \ge 0} J_\nu(\alpha) \right).$

In Section \ref{sec_multi_decomp}, we will show that each set $J_\nu(\alpha)$, if it is non-empty, is dense in $X$. Moreover, the set $J_\nu'$ is non-empty.

\section{Pressure and temperature} \label{sec_pres_temp}
In this section, we define and prove basic properties of the temperature function. Since the temperature function is defined using the pressure of a 2-parameter family of potentials, we first define this 2-parameter family of potentials and prove basic properties of the pressure function in Section \ref{sec_3.2_pressure}. We discuss the temperature function in Section \ref{sec_temp}.
Throughout this section, we fix a holomorphic endomorphism $f \colon \mathbb P^k \to \mathbb P^k$ of algebraic degree at least 2 and an expanding set $X \subseteq \mathbb P^k$.

\subsection{The pressure function} \label{sec_3.2_pressure}
We begin by defining a 2-parameter family of potential functions that we will work with. 
Let $g \colon X \to \mathbb R$ be a H\"older continuous function. For $(q,t) \in \mathbb R^2$, we consider the following 2-parameter family of potentials
\begin{equation} \label{eq_potential_qt}
\phi_{q,t} := qg - t\log |\Jac f|.
\end{equation}
Since $\phi_{q,t}$ is H\"older continuous, we denote by $\nu_{q,t}$ its unique equilibrium state.

Given a H\"older continuous function $\phi \colon X \to \mathbb R$, recall from \cite{BHMane} the definition of the pressure $\mathcal{P}_X^+(\phi)$ of $\phi$
$$\mathcal{P}_X^+(\phi) := \sup_{\nu \in \mathcal{M}_X^+(f)} \left\{h_{\nu}(f) + \int \phi d\nu \right\} $$
where the sup is taken over all $\nu \in \mathcal{M}_X^+(f)$. By the variational principle, we have $\mathcal{P}_X^+(\phi) =\mathcal{P}_{X}^{\rm top}(\phi)$, where $\mathcal{P}_{X}^{\rm top}(\phi)$ is the {\it topological pressure} that we recall now from \cite[Chapter 11]{Urbanski21}. 
Consider a sequence $\{\cV_n\}_{n\in\NN}$ of open finite covers of $X$ such that $\lim_{n\to\infty}\diam(\cV_n) = 0.$
Then the topological pressure is given by 
$$
    \cP_X^{\rm top}(\phi) := \lim_{n\to\infty}P(\phi,\cV_n) = \lim_{n\to\infty}\lim_{m\to\infty}\frac{1}{m}\log\inf_{\cV_n}\left\{\sum_{V\in \cV_n}\exp(\overline{S_m}\phi(V))\right\}
$$
where $\overline{S_m}\phi(V) := \sup\left\{\sum_{k=0}^{n-1}\phi\circ f^k(x): x\in V\right\}$.

\smallskip

Recall the 2-parameter family of potential functions $\phi_{q,t}$ in \eqref{eq_potential_qt}. Up to replacing $g$ by $g - \mathcal{P}_X^+(g)$, we can assume that $\mathcal{P}_X^+(g)=0$.
With this assumption, we set $$\mathcal{P}(q,t) :=\mathcal{P}_X^+(\phi_{q,t}).$$

\begin{lemma} \label{lem_property_qt}
The pressure function $(q,t) \mapsto \mathcal{P}(q,t)$ satisfies the following properties.
\begin{enumerate}
\item For any $(q,t) \in \mathbb R^2$, we have $\mathcal{P}(q,t) < \infty$. \label{ITM:FinP}
\item The function $(q,t) \mapsto \mathcal{P}(q,t)$ is real-analytic. \label{ITM:ana}
\item For each fixed $q \in \mathbb R$, the map $t \mapsto \mathcal{P}(q,t)$ is strictly decreasing. Moreover, we have $\lim_{t \to -\infty}\mathcal{P}(q,t) = \infty$ and $\lim_{t \to \infty}\mathcal{P}(q,t) = -\infty$. \label{ITM:SDP}
\item  $\frac{\partial \mathcal{P}}{\partial t} (q,t) = -\int\log|\Jac f|d\nu_{q,t} < 0$. 
\label{ITM:tDerivative}
\item  $\frac{\partial \mathcal{P}}{\partial q}(q,t) = \int g\; d\nu_{q,t}$. \label{ITM:qDerivative}
\item If $g<0$, then we have $\mathcal{P}(q_1,t_1) \le \mathcal{P}(q_2,t_2)$ if $q_1 \ge q_2$ and $t_1 \ge t_2$. \label{ITM:tq-mono}
\end{enumerate}
\end{lemma}

\begin{proof} 
To see (\ref{ITM:FinP}), note that $(X,f)$ is an expansive topological dynamical system in the sense of Urbanski-Roy-Munday \cite{Urbanski21}. By \cite[Proposition 11.3.2]{Urbanski21}, for each $(q,t)\in \mathbb R^2$, we have  
$$
    \cP(q,t)\leq h_{top}(f) + \sup_{x\in X} \phi_{q,t}.
$$
As $\phi_{q,t}$ is H\"older continuous on $X$, we have $|\sup_{x \in X}\phi_{q,t}|<\infty$. Hence (1) follows. 

\smallskip

For \eqref{ITM:ana}, by \cite[Theorem 13.10.3]{Urbanski22} which shows that the mapping of a H\"older continuous potential $\phi:X\to\RR$ with exponent $\beta$ to its topological pressure $\cP_X^{\rm top}(\phi)$, 
is real analytic and the fact that $(q,t)\mapsto\phi_{q,t}$ is affine, we see that the composition $(q,t) \mapsto \mathcal{P}(q,t)$ is real-analytic. 

\smallskip

For (\ref{ITM:SDP}), let $q\in\RR$ be fixed and let $t'>t$. Then by the variational principal, we have
\begin{align}
    \cP(q,t)&\geq h_{\nu_{q,t'}}(f)+q\int g\;d\nu_{q,t'}-t\int\log|\Jac f|d\;\nu_{q,t'}\notag\\
    &>h_{\nu_{q,t'}}(f)+q\int g\;d\nu_{q,t'}-t'\int\log|\Jac f|d\;\nu_{q,t'} = \cP(q,t').
\end{align}
Therefore $t \mapsto \mathcal{P}(q,t)$ is strictly decreasing. To see the last two statements, we write
$$
    \cP(q,t) = h_{\nu_{q,t}}(f)+q\int g d\nu_{q,t}-t\int\log|\Jac f|d\nu_{q,t}.
$$
Taking limits as $t\to \infty$ and $t\to -\infty$ then yields the result.

\smallskip

For (\ref{ITM:tDerivative}), the argument is near identical to \cite[Theorem 16.4.10]{Urbanski22}. Using the variational principal, 
\begin{align*}
\cP(q,t')&\geq h_{\nu_{q,t}}(f)+q\int g\;d\nu_{q,t}-t'\int\log|\Jac f|d\;\nu_{q,t}\notag\\
& =\cP(q,t)+(t-t')\int \log |\Jac f|d\nu_{q,t}.
\end{align*}
Therefore,
$$
\cP(q,t')\geq \cP(q,t)+(t-t')\int\log|\Jac f|\;d\nu_{q,t}.
$$
If $t'<t$, then $t'-t<0$ and 
$$
    \frac{\cP(q,t')-\cP(q,t)}{t'-t}\leq -\int\log|\Jac f|\;d\nu_{q,t}<0.
$$
Similarly, if $t'>t$, then $t'-t>0$ and 
$$
    \frac{\cP(q,t')-\cP(q,t)}{t'-t}\geq -\int\log|\Jac f|\;d\nu_{q,t}.
$$
Taking limits and noting that $\cP(q,t)$ is real-analytic (and thus differentiable) finishes the argument. The proof of (\ref{ITM:qDerivative}) is similar. The reader should note that the computation does not result in a negative sign in front of the integral in this case. 

\smallskip

For (\ref{ITM:tq-mono}), we follow \cite[Section 21.1]{Urbanski22}. Let $q_1\geq q_2$ and $t_1\geq t_2$. Then if $g<0$, $\phi_{q_1,t_1}(x)\leq \phi_{q_2,t_2}(x)$ for all $x\in X$. Thus, for any open set $V$, $\overline{S_m}\phi_{q_1,t_1}(V) \leq \overline{S_m}\phi_{q_2,t_2}(V)$. Since $\exp(\cdot)$ and $\log(\cdot)$ are monotone increasing, and $\inf\{\cdot\}$ preserves inequalities, the result follows. 
\end{proof}

\subsection{The temperature function}\label{sec_temp}
In this section, we define the {\it temperature function} and study its properties. We continue to use notations from Section \ref{sec_3.2_pressure}. Since $X$ is uniformly expanding, the arguments in this section are standard; see for example \cite{Barreirabook11,Pesin97book,PU}.

\begin{definition}
The {\it temperature function} $T \colon \mathbb R \to \mathbb R$ is defined as $q \mapsto T(q)$, where $T(q)$ is the unique number such that
$\mathcal{P}(q,T(q)) = 0$.
\end{definition}

\begin{remark}\label{rmk_T(0)}
By Lemma \ref{lem_property_qt} \eqref{ITM:SDP}, the above definition is well-posed. Moreover, we have $T(0) = \VD_{f}(X)$ by \cite[Theorem 5.1]{BHMane}.
\end{remark}

\begin{lemma}\label{lem_temp_an}
The temperature function $T \colon \mathbb R \to \mathbb R$ is real-analytic. 
\end{lemma}
\begin{proof}  

Since the function $(q,t)\mapsto \mathcal{P}(q,t)$ is real-analytic by Lemma \ref{lem_property_qt} \eqref{ITM:ana},
the conclusion follows from the Inverse Function Theorem once we establish
$$\frac{\partial\mathcal{P}(q,t)}{\partial t}\bigg|_{(q_0,t_0)} \neq 0 \text{ for any } (q_0,t_0) \in \mathbb R^2.$$
By Lemma \ref{lem_property_qt} \eqref{ITM:tDerivative}, we have
$$\frac{\partial\mathcal{P}(q,t)}{\partial t}\bigg|_{(q_0,t_0)} = -\int_{X}\log |\Jac f|d\nu_{q_0,t_0} \neq 0$$
where $\nu_{q_0,t_0}$ is the equilibrium state for $\phi_{q_0,t_0}$. This completes the proof.
\end{proof}

Fix $q \in \mathbb R$. Let $\nu_q$ be the equilibrium state of $\phi_{q,T(q)}$. We define 
\begin{equation} \label{eq_def_alpha_q}
\alpha(q) := -\frac{\int_X g d\nu_q}{\int_X \log |\Jac f| d\nu_q} >0.
\end{equation}

\begin{lemma} \label{lem_temp_de}
The temperature function $T \colon \mathbb R \to \mathbb R$ satisfies $T'(q) = -\alpha(q)$ for all $q \in \mathbb R$. In particular, it is strictly decreasing.
\end{lemma}
\begin{proof}
Since $\mathcal{P}(\phi_{q,T(q)})=0$ for any $q \in \mathbb R$, we have
$$\frac{d}{dq}\mathcal{P}(\phi_{q,T(q)}) = \frac{\partial\mathcal{P}(q,t)}{\partial q} + \frac{\partial\mathcal{P}(q,t)}{\partial t}\bigg|_{t = T(q)}T'(q) = 0.$$
Therefore, we have
$$T'(q) = - \frac{\partial\mathcal{P}(q,t)}{\partial q} \cdot \left(\frac{\partial\mathcal{P}(q,t)}{\partial t}\bigg|_{t = T(q)} \right)^{-1}.$$
By Lemma \ref{lem_property_qt} \eqref{ITM:qDerivative}, we have
$$\frac{\partial\mathcal{P}(q,t)}{\partial q}\bigg|_{(q_0,t_0)} = \int_{X}g d\nu_{q_0,t_0} \neq 0$$
where $\nu_{q_0,t_0}$ is the equilibrium state for $\phi_{q_0,t_0}$.

Hence, we obtain
$$T'(q_0) = -\frac{\int_{X}g d\nu_{q_0,t_0}}{-\int_{X} \log|\Jac f| d\nu_{q_0,t_0}} =-\alpha(q_0).$$
The conclusion follows. 
\end{proof}

Observe that $\nu_0$ is the unique equilibrium state of the potential $\phi_{0,T(0)} = -2\VD_{f}(X)\log|\Jac f|$.
\begin{lemma} \label{lem_temp_conv}
The temperature function $T \colon \mathbb R \to \mathbb R$ is convex. Moreover, it is strictly convex if and only if $\nu \neq \nu_0$.
\end{lemma}
\begin{proof}
Differentiating $\mathcal{P}(\phi_{q,T(q)})$ with respect to $q$ twice, we obtain
$$T''(q) = -\frac{T'(q)^2 \cdot \frac{\partial^2\mathcal{P}(q,t)}{\partial t^2} + 2T'(q)\cdot \frac{\partial^2\mathcal{P}(q,t)}{\partial q \partial t} + \frac{\partial^2\mathcal{P}(q,t)}{\partial q^2} }{\frac{\partial\mathcal{P}(q,t)}{\partial t}}$$
where the partial derivatives are evaluated at $(q,t) = (q,T(q))$. 

\begin{align*}
\frac{\partial^2}{\partial t^2}\mathcal{P}(-t\log|\Jac f|+qg) &= \frac{\partial^2}{\partial t^2}\mathcal{P}(-(t_0+\epsilon_1+\epsilon_2)\log|\Jac f|+qg) \\
&={\rm Var}_{\phi_q}(\log|\Jac f|, \log|\Jac f|).
\end{align*}
Here the (asymptotic) variance ${\rm Var}$ is given by
$${\rm Var}_h(h_1,h_2) := \sum_{k=0}^\infty \left(\int_X h_1(h_2 \circ f^k)d\mu_h - \int_X h_1d\mu_h \int_X h_2d\mu_h  \right)$$
where $\mu_h$ is the equilibrium state for $h$.

Similarly, we have
$$\frac{\partial^2}{\partial q^2}\mathcal{P}(-t\log|\Jac f|+qg) ={\rm Var}_{\phi_q}(g, g), \text{ and }$$
$$\frac{\partial^2}{\partial q \partial t}\mathcal{P}(-t\log|\Jac f|+qg) ={\rm Var}_{\phi_q}(g, \log|\Jac f|).$$

Hence, we obtain
$$T''(q) = \frac{{\rm Var}_{\phi_q}(g-T'(q)\log|\Jac f|, g-T'(q)\log|\Jac f|)}{\int_X \log|\Jac f|d\nu_q}.$$
Therefore, $T''(q) \ge 0$ for any $q \in \mathbb R$. Moreover, $T''(q) = 0$ if and only if the function $g-T'(q)\log|\Jac f|$ is cohomologous to a constant, if and only if $g$ and $-T'(q)\log|\Jac f|$ are cohomologous. If $g$ and $-T'(q)\log|\Jac f|$ are cohomologous, then $\mathcal{P}(-T'(q)\log|\Jac f|) = \mathcal{P}(g) =0$, which implies $-T'(q) = \VD_f(X)$ and thus $\nu = \nu_0$.
\end{proof}

\begin{corollary}\label{cor_alpha_q}
If $\nu \neq \nu_0$, then the function $q \mapsto \alpha(q)$ is strictly decreasing. If $\nu = \nu_0$, then $\alpha(q) = \VD_f(X)$ for all $q \in \mathbb R$.
\end{corollary}
\begin{proof}
Since $\alpha(q) = -T'(q)$ by Lemma \ref{lem_temp_de}, the conclusion follows from Lemma \ref{lem_temp_conv}.
\end{proof}

\begin{remark}\label{rmk_range_alpha(q)}
We denote by $[\alpha_1,\alpha_2] \subset (0,\infty)$ the range of the function $q \mapsto \alpha(q)$. By Corollary \ref{cor_alpha_q}, if $\nu = \nu_0$, then $\alpha_1=\alpha_2=\VD_f(X)>0$. If $\nu \neq \nu_0$, then $\alpha_1 = \al(\infty)$ and $\alpha_2 = \al(-\infty)$. In this case, $\alpha_2 = \alpha(-\infty)<\infty$ since $\lim_{q \to -\infty} \frac{-\int g d\nu_q}{\int \log |\Jac f|d\nu_q} \le \frac{\sup(-g)}{\inf \log|\Jac f|} < \infty$. Moreover, $\alpha_1 = \alpha(\infty)>0$ since $\frac{-\int g d\nu_q}{\int \log |\Jac f|d\nu_q} \ge |T(q_0)|/q_0>0$ for all $q$ where $q_0>0$ is a fixed number with $T(q_0)<0$; see \cite[p. 254]{PU}.
\end{remark}

\section{Proof of Theorem \ref{thm_main_1}}\label{sec_main_1}

Let $f \colon \mathbb P^k \to \mathbb P^k$ be a holomorphic endomorphism of algebraic degree at least 2, $X \subseteq \mathbb P^k $ an uniformly expanding set, $g \colon X \to \mathbb R$ a H\"older continuous function and $\nu \in \mathcal{M}_X^+(f)$ the equilibrium state of $g$.
For each $q \in \mathbb R$, recall the potential function $\phi_q = \phi_{q,T(q)} \colon X \to \mathbb R$ given by
$$\phi_q := qg -T(q) \log|\Jac f|$$ and its unique equilibrium state $\nu_q \in \mathcal{M}_X^+(f)$. Recall also $\alpha(q)$ as defined in \eqref{eq_def_alpha_q} and
$J_\nu(\alpha(q))$ as defined in \eqref{Def:Jal}.

\begin{definition} \label{def_s_nu_alpha_q}
The {\it parametrized volume dimension multifractal spectrum} is the map $\hat{S}_\nu \circ \alpha \colon \mathbb R \to \mathbb R$ given by
\begin{equation*}
\hat{S}_\nu(\alpha(q)) := \VD_{f,\nu_q}(J_\nu(\alpha(q))).
\end{equation*}
\end{definition}

\begin{remark}
We sometimes need to regard $S_\nu(\alpha(q))$ as a function of $\alpha$. We write $S_\nu(\alpha)$ for $\alpha \in [\alpha_1,\alpha_2]$ which is the range of $\alpha(q), q\in\mathbb R$. In this case, provided $\nu \neq \nu_0$, $\hat{S}_\nu(\alpha) = \VD_{f,\nu_q}(J_\nu(\alpha(q)))$ where $q$ is the unique real number such that $\alpha(q) = \alpha$. The uniqueness of $q$ is guaranteed by Corollary \ref{cor_alpha_q}.
\end{remark}

\subsection{Symbolic coding}
Since $f \colon X \to X$ is uniformly expanding, there exists a {\it Markov partition} $\mathcal{R} :=\{R_1,\ldots, R_p\}$ of $X$ so that the dynamical system $f \colon X \to X$ is conjugate to a one-sided subshift of finite type $\sigma \colon \Sigma_A^+ \to \Sigma_A^+$. We denote by $\pi \colon \Sigma_A^+ \to X$ the projection map. 

If $\phi \colon X \to \mathbb R$ is a H\"older continuous function, then we denote by $\widetilde{\phi}:= \phi \circ \pi$, and it is also H\"older continuous. Both $\phi$ and $\widetilde{\phi}$ have a unique equilibrium state, denote by $\nu_\phi$ and $\widetilde{\nu}_{\widetilde{\phi}}$ respectively, and they satisfy $\nu_\phi = \pi_* \widetilde{\nu}_{\widetilde{\phi}}$.

If $\mu$ is the equilibrium state of a H\"older continuous potential $\phi \colon X\to \RR$, then $\mu$ satisfies the {\it Gibbs property} (see \cite[Section 13.2]{Urbanski22}). That is, for every $\omega \in \Sigma_A^+,\epsilon>0,0<\kappa<r(\ep)$ and $m >n(\ep)$, we have
\begin{equation}\label{eq_gibbs_nu}
C_1 \le \frac{\nu(U(\pi(\omega),\ep,\kappa,m))}{\prod_{k=0}^{m-1} e^{\phi(\pi(\sigma^k(\omega)))}} \le C_2
\end{equation}
where $C_1,C_2>0$ are constants. In particular, $C_1$ and $C_2$ may be taken such that $C_2>1$ and $C_1 = C_2^{-1}$.
Recall $r(\ep),n(\ep)$ from Section \ref{sec_lvd}. 

\smallskip

Given $\ep>0,\kappa>0,N>0$, let $r = r(\ep,\kappa,N)>0$ be a small positive real number depending on $\ep,\kappa$ and $N$. For $\omega \in \Sigma_A^+$, let $m(\omega) = m(\omega,\ep,\kappa,N)$ be the unique positive integer such that
\begin{align} \label{eq_vol}
&\prod_{k=0}^{m(\omega)} \left|\Jac f (\pi(\sigma^k(\omega)))\right|^{-1}> r 
\text{ and } \\ &\prod_{k=0}^{m(\omega)+1} \left|\Jac f (\pi(\sigma^k(\omega)))\right|^{-1} \le r. 
\end{align}
In our applications, $r$ will be  
${\rm Vol}(U(x,\epsilon,\kappa,N))^{1/2}$ for a fixed $x \in X$.

\begin{lemma}\label{lem_4.3}
Fix $x \in X$ and $\phi \colon X\to \RR$ a H\"older continuous potential with equilibrium state $\mu$. Further, let $\omega \in \Sigma_A^+$ be such that $\pi(\omega)=x$ and for given $\epsilon>0,0<\kappa<r(\ep),N>n(\ep)$, set $r = {\rm Vol}(U(x,\ep,\kappa,N))^{1/2}$.
Then we have
$$\lim_{\ep \to 0}\lim_{\kappa \to 0}\lim_{N \to \infty} \frac{\log \mu(U(x,\ep,\kappa, m(\omega)))}{\log r} = \lim_{\ep \to 0}\lim_{\kappa \to 0}\lim_{N \to \infty} \frac{\log \mu(U(x,\ep,\kappa,N))}{\log r}$$
should either limit exist.
\end{lemma}
\begin{proof}

Fix $\epsilon>0$ and $0<\kappa<r(\epsilon)$. 
Note that by inequality \eqref{eq_vol}, \cite[Corollary 2.4 and Lemma 2.7]{BHMane}, we have

\begin{equation}\label{ineq:mwgeq}
    m(\om) \geq -\frac{k\log\kappa}{L_{\nu}+k\epsilon}+N\left(\frac{L_{\nu}-k\epsilon}{L_{\nu}+k\epsilon}\right)-1.
\end{equation}

For all $N>n(\ep)$ sufficiently large, note that 
\begin{equation}\label{ineq:mN1}
\lim_{\ep \to 0}\lim_{\kappa \to 0}\lim_{N \to \infty} \frac{\log \mu(U(x,\ep,\kappa, m(\omega)))}{\log r} \geq \lim_{\ep \to 0}\lim_{\kappa \to 0}\lim_{N \to \infty} \frac{\log \mu(U(x,\ep,\kappa,N))}{\log r} 
\end{equation}
is immediate as $m(\om)>N$ by \eqref{ineq:mwgeq} and the $U$ sets are nested. Now we show the other direction. 
By the Gibbs property of $\mu$ \eqref{eq_gibbs_nu}, we have
\begin{align*}
    \log&\;\mu(U(x,\epsilon,\kappa,m(\om))) \geq \log C_1 + \sum_{k=0}^{m(\om)-1} \phi(\pi(\sigma^k(\omega)))\\
    & = \log C_1 -\log C_2 +\log C_2+\sum_{k=0}^{m(\om)-1} \phi(\pi(\sigma^k(\omega)))\\
    & \ge \log \frac{C_1}{C_2} + \log \mu(U(x,\epsilon,\kappa,N)) + \sum_{k=N}^{m(\omega)-1}\phi(\pi(\sigma^k(\omega)))\\
    & \ge \log \frac{C_1}{C_2} + \log \mu(U(x,\epsilon,\kappa,N)) + (m(\omega)-N)\inf \phi.
\end{align*}
So for all N large enough, we have
\begin{align}\label{eq_4.3_1}
    \frac{\log \mu(U(x,\epsilon,\kappa,m(\om)))}{\log r} &\leq \frac{\log \mu(U(x,\epsilon,\kappa,N))}{\log r}+\frac{\log\frac{C_1}{C_2}}{\log r}+\left(\frac{m(\om)-N}{\log r}\right)\inf \phi. 
\end{align}
Consider the term $\left(\frac{m(\om)-N}{\log r}\right)$. By \eqref{ineq:mwgeq} and \cite[Lemma 2.7]{BHMane}, we obtain

\begin{equation}\label{asymp_mw}
    \left(\frac{m(\om)-N}{\log r}\right)\leq \frac{-\frac{k\log\kappa}{L_{\nu}+k\epsilon}+N\left(\frac{L_{\nu}-k\epsilon}{L_{\nu}+k\epsilon}-1\right)-1}{k\log\kappa -N(L_{\nu}+k(2M+2)\epsilon)}.
\end{equation}
And so we have 
\begin{align*}
    \lim_{\ep \to 0}\lim_{\kappa \to 0}\lim_{N \to \infty}\left(\frac{m(\om)-N}{\log r}\right) = \lim_{\ep \to 0}\lim_{\kappa \to 0}-\left(\frac{\frac{L_{\nu}-k\epsilon}{L_{\nu}+k\epsilon}-1}{L_{\nu}+k(2M+2)\epsilon}\right)=0.
\end{align*}
Hence, taking limits $\lim_{\ep \to 0}\lim_{\kappa \to 0}\lim_{N \to \infty}$ on both sides of \eqref{eq_4.3_1} gives
\begin{equation}\label{ineq:mN2}
\lim_{\ep \to 0}\lim_{\kappa \to 0}\lim_{N \to \infty} \frac{\log \mu(U(x,\ep,\kappa, m(\omega)))}{\log r} \leq \lim_{\ep \to 0}\lim_{\kappa \to 0}\lim_{N \to \infty} \frac{\log \mu(U(x,\ep,\kappa,N))}{\log r}.
\end{equation}
The result then follows from \eqref{ineq:mN1} and \eqref{ineq:mN2}.
\end{proof}

\subsection{Key proposition}
In this section, we prove Proposition \ref{prop_lem3}, which are key ingredients needed in the proof of Theorem \ref{thm_main_1}.

For $q \in \mathbb R$,

we define the {\it symbolic level set} $\hat{X}_{q} \subseteq \Sigma_A^+$ as 
\begin{equation}\label{eq_def_symbolic_level}
\hat{X}_{q} :=\left\{\omega \in \Sigma_A^+ : 
\lim_{n\to\infty}\frac{\sum_{k=0}^{n-1} g(\pi(\sigma^k(\omega)))}{\sum_{k=0}^{n-1}  \log|\Jac f(\pi(\sigma^k(\omega)))|^{-1}} = \alpha(q) \right\}.
\end{equation}

\begin{lemma}\label{lem_2}
For every $q \in \mathbb R$, we have $\nu_q(\pi(\hat{X}_q)) =1$.
\end{lemma}
\begin{proof}
Denote by $\widetilde{\nu}_q$ the equilibrium state of $\widetilde{\phi}_q$. Since $\widetilde{\nu}_q$ is ergodic, by the Birkhoff ergodic theorem, for $\widetilde{\nu}_q$-a.e. $\omega \in \Sigma_A^+$, we have
\begin{align}\label{eq_limit}
\lim_{n \to \infty} \frac{\sum_{k=0}^{n-1} (g\circ \pi)(\sigma^k(\omega))}{\sum_{k=0}^{n-1} \log|\Jac f(\pi(\sigma^k(\omega)))|^{-1}} &= \frac{\int_{\Sigma_A^+} g\circ \pi ~d\widetilde{\nu}_q}{\int_{\Sigma_A^+} \log |(\Jac f)\circ \pi|^{-1} ~d\widetilde{\nu}_q} = \frac{\int_{X} g ~d\nu_q}{\int_{X} \log|\Jac f|^{-1} ~d\nu_q} = \alpha(q). \nonumber
\end{align}
The conclusion follows.
\end{proof}

Now we prove the key proposition. 
\begin{proposition}\label{prop_lem3}
For every $q \in \mathbb R$, we have
\begin{enumerate}
\item $\nu_q(J_{\nu}(\alpha(q))) =1$;
\item $\delta_{\nu_q,x} = T(q)+q\alpha(q)$ for every $x \in J_{\nu}(\alpha(q))$; 
\item $\VD_{f,\nu_q}(J_{\nu}(\alpha(q))) = T(q)+q\alpha(q)$. 
\end{enumerate}
\end{proposition}

The proof of Proposition \ref{prop_lem3} consists of Lemmas \ref{lem_3_part1} and \ref{lem_3_part2}.

\begin{lemma}\label{lem_3_part1}
For each $q \in \mathbb R$, we have $\pi(\hat{X}_q) = J_{\nu}(\alpha(q))$. In particular, $\nu_q(J_{\nu}(\alpha(q)))=1$. 
\end{lemma}
\begin{proof}
We first note that once we establish $\pi(\hat{X}_q) \subseteq J_{\nu}(\alpha(q))$, then $\nu_q(J_{\nu}(\alpha(q)))=1$ follows from Lemma \ref{lem_2}.

Now we show that $\pi(\hat{X}_q) \subseteq J_{\nu}(\alpha(q))$. 
Let $\omega \in \hat{X}_q$. By definition of $\hat{X}_q$, we have the limit 
\begin{equation}\label{eq_4.9_limit}
\lim_{n \to \infty} \frac{\sum_{k=0}^{n-1} g(\pi(\sigma^k(\omega)))}{\sum_{k=0}^{n-1} \log  |\Jac f(\pi(\sigma^k(\omega)))|^{-1}} = \alpha(q).
\end{equation}

We need to show $\pi(\omega) \in J_\nu(\alpha(q))$, i.e., $\delta_{\nu,\pi(\omega)} = \alpha(q)$. We first prove $\alpha(q) \le \delta_{\nu,\pi(\omega)}$.
Fix $\ep>0$. For any $0<\kappa<r(\ep),N>n(\ep)$, set $r = {\rm Vol}(U(\pi(\om),\ep,\kappa,N))^{1/2}$. By definition of $m(\omega,\ep,\kappa,N)$ in  \eqref{eq_vol} and
Gibbs property of $\nu$ \eqref{eq_gibbs_nu}, we have
\begin{align}\label{eq_4.10_1}
\frac{\sum_{k=0}^{m(\omega)-1} g(\pi(\sigma^k(\omega)))}{\sum_{k=0}^{m(\omega)-1}  \log |\Jac f(\pi(\sigma^k(\omega)))|^{-1}} 
& \le \frac{-\log C_2 + \log \nu(U(\pi(\omega),\ep,\kappa,m(\om)))}{\log {\rm Vol}(U(\pi(\omega),\ep,\kappa,N))^{1/2}}.
\end{align}

Take limits as $\ep \to 0, \kappa \to 0, N \to \infty$ on both sides of \eqref{eq_4.10_1}. For the left hand side, since equation \eqref{eq_4.9_limit} holds, we see that
$$\lim_{\ep \to 0}\lim_{\kappa \to 0} \lim_{N \to \infty} \frac{\sum_{k=0}^{m(\omega)-1} g(\pi(\sigma^k(\omega)))}{\sum_{k=0}^{m(\omega)-1} \log |\Jac f(\pi(\sigma^k(\omega)))|^{-1}} = \alpha(q).$$
For the right hand side, by Lemma \ref{lem_4.3},
we have
$$\lim_{\ep \to 0}\lim_{\kappa \to 0} \lim_{N \to \infty}\frac{-\log C_2 + \log \nu(U(\pi(\omega),\ep,\kappa,m(\om)))}{\log {\rm Vol}(U(\pi(\omega),\ep,\kappa,N))^{1/2}}=\delta_{\nu,\pi(\omega)}.$$
Therefore, we have
\begin{equation*}
\alpha(q) \le \delta_{\nu,\pi(\omega)}.
\end{equation*}

\medskip

Now we prove $\alpha(q) \ge \delta_{\nu,\pi(\omega)}$. 
Fix $\ep>0$. For any $0<\kappa<r(\ep),N>n(\ep)$, set $r = {\rm Vol}(U(\pi(\om),\ep,\kappa,N))^{1/2}$.
Again, by definition of $m(\omega,\ep,\kappa,N)$ in \eqref{eq_vol} and Gibbs property of $\nu$ \eqref{eq_gibbs_nu}, we have
\begin{align}\label{eq_4.10_2}
\frac{\sum_{k=0}^{m(\omega)-1} g (\pi(\sigma^k(\omega)))}{\sum_{k=0}^{m(\omega)+1}  \log |\Jac f(\pi(\sigma^k(\omega)))|^{-1}} 
& \ge \frac{-\log C_1  + \log \nu(U(\pi(\omega),\ep,\kappa,m(\om)))}{\log {\rm Vol}(U(\pi(\omega),\ep,\kappa,N))^{1/2}}.
\end{align}

Take limits as $\ep \to 0, \kappa \to 0, N \to \infty$ on both sides of \eqref{eq_4.10_2}. For the left hand side, since equation \eqref{eq_4.9_limit} holds, we see that
$$\lim_{\ep \to 0}\lim_{\kappa \to 0} \lim_{N \to \infty} \frac{\sum_{k=0}^{m(\omega)-1} g(\pi(\sigma^k(\omega)))}
{\sum_{k=0}^{m(\omega)+1}  \log |\Jac f(\pi(\sigma^k(\omega)))|^{-1}} = \alpha(q).$$
For the right hand side, by Lemma \ref{lem_4.3}, we have
$$\lim_{\ep \to 0}\lim_{\kappa \to 0} \lim_{N \to \infty}\frac{-\log C_1 +\log \nu(U(\pi(\omega),\ep,\kappa,m(\om)))}{\log {\rm Vol}(U(\pi(\omega),\ep,\kappa,N))^{1/2}}=\delta_{\nu,\pi(\omega)}.$$

Therefore, we have
\begin{equation*}
\alpha(q) \ge \delta_{\nu,\pi(\omega)}.
\end{equation*}
Hence, $\pi(\om) \in J_\nu(\alpha(q))$. This completes the proof of $\pi(\hat{X}_q) \subseteq J_{\nu}(\alpha(q))$.

Now we show $J_{\nu}(\alpha(q)) \subseteq \pi(\hat{X}_q)$. Let $x \in J_{\nu}(\alpha(q))$ and let $\omega \in \Sigma_A^+$ be such that $\pi(\om) =x$. We need to show that $\omega \in \hat{X}_q$. To this end, by the same calculation as in \eqref{eq_4.10_1} and \eqref{eq_4.10_2}, it suffices to show that
$$\lim_{\ep \to 0}\lim_{\kappa \to 0} \lim_{N \to \infty} \frac{\sum_{k=0}^{m(\omega)-1} g(\pi(\sigma^k(\omega)))}
{\sum_{k=0}^{m(\omega)+1}  \log |\Jac f(\pi(\sigma^k(\omega)))|^{-1}} = \lim_{\ep \to 0}\lim_{\kappa \to 0} \lim_{N \to \infty} \frac{\sum_{k=0}^{N-1} g(\pi(\sigma^k(\omega)))}
{\sum_{k=0}^{N-1}  \log |\Jac f(\pi(\sigma^k(\omega)))|^{-1}}.$$

Note that \eqref{asymp_mw} may be extended to
\begin{equation}\label{asymp_mw2}
   \frac{-\frac{k\log\kappa}{L_{\nu}-k\epsilon}+N\left(\frac{L_{\nu}+k(2M+2)\epsilon}{L_{\nu}-k\epsilon}-1\right)}{k\log\kappa -N(L_{\nu}-k\epsilon)} \leq\left(\frac{m(\om)-N}{\log r}\right)\leq \frac{-\frac{k\log\kappa}{L_{\nu}+k\epsilon}+N\left(\frac{L_{\nu}-k\epsilon}{L_{\nu}+k\epsilon}-1\right)-1}{k\log\kappa -N(L_{\nu}+k(2M+2)\epsilon)}.
\end{equation}
Now we have
\begin{align*}
\frac{\sum_{j=0}^{m(\om)}g(f^j(x))}{a_N+\sum_{j=0}^{m(\om)}\log|\Jac f(f^j(x))|^{-1}}
\le \frac{\sum_{j=0}^{N}g(f^j(x))}{\sum_{j=0}^{N}\log|\Jac f(f^j(x))|^{-1}}
\le \frac{b_N+\sum_{j=0}^{m(\om)}g(f^j(x))}{\sum_{j=0}^{m(\om)}\log|\Jac f(f^j(x))|^{-1}}
\end{align*}
where $a_N = -(m(\om)-N)\sup(\log|\Jac f|^{-1})$ and $b_N =-(m(\om)-N)\inf g$. The left most side of this inequality may be written as 
$$
    \frac{\frac{\sum_{j=0}^{m(\om)}g(f^j(x))}{\sum_{j=0}^{m(\om)}\log|\Jac f(f^j(x))|^{-1}}}{\frac{a_N}{\sum_{j=0}^{m(\om)}\log|\Jac f(f^j(x))|^{-1}}+1}
$$
and the right most side of this inequality may be written as
$$
    \frac{b_N}{\sum_{j=0}^{m(\om)}\log|\Jac f(f^j(x))|^{-1}}+\frac{\sum_{j=0}^{m(\om)}g(f^j(x))}{\sum_{j=0}^{m(\om)}\log|\Jac f(f^j(x))|^{-1}}
$$

Now by \eqref{eq_vol} and \eqref{asymp_mw2} we have that 
$$
\lim_{\epsilon\to0}\lim_{\kappa\to0}\lim_{N\to\infty}\frac{a_N}{\sum_{j=0}^{m(\om)}\log|\Jac f(f^j(x))|^{-1}}=0
$$
and
$$
\lim_{\epsilon\to0}\lim_{\kappa\to0}\lim_{N\to\infty}\frac{b_N}{\sum_{j=0}^{m(\om)}\log|\Jac f(f^j(x))|^{-1}} = 0.
$$
The conclusion follows. This completes the proof.
\end{proof}

\begin{remark}\label{rmk_Jnu_invariant}
Since $\hat{X}_q$ is $\sigma$-fully-invariant by Birkhoff's ergodic theorem, by the above lemma, we see that $J_{\nu}(\alpha(q))$ is $f$-fully-invariant.
\end{remark}

\begin{lemma}
\label{lem_3_part2}
For every $q \in \mathbb R$, we have
\begin{enumerate}
\item $\delta_{\nu_q,x} = T(q)+q\alpha(q)$ for every $x \in J_{\nu}(\alpha(q))$\label{nu_q_local_dim}; 
\item $\VD_{f,\nu_q}(J_{\nu}(\alpha(q))) = T(q)+q\alpha(q)$. 
\end{enumerate}
\end{lemma} 
\begin{proof}
We first prove statement (1). Fix $x \in J_{\nu}(\alpha(q))$. Let $\om \in \Sigma_A^+$ be such that $\pi(\om) = x$. Fix $\ep>0$. For any $0<\kappa<r(\ep),N>n(\ep)$, set $r = {\rm Vol}(U(\pi(\om),\ep,\kappa,N))^{1/2}$.

By Gibbs property of $\nu$ \eqref{eq_gibbs_nu} and that $g$ is bounded on the compact set $X$, we have
\begin{equation}\label{eq_4.12_1}
C_1 \le \frac{\nu(U(x,\ep,\kappa,m(\omega)))}{\prod_{k=0}^{m(\omega)} e^{g(\pi(\sigma^k(\omega)))}} \le C_2.
\end{equation}

By Gibbs property of $\nu_q$ and that $\phi_{q}$ is bounded on compact $X$, we have
\begin{equation}\label{eq_4.12_2}
C_1' \le \frac{\nu_q(U(x,\ep,\kappa,m(\omega)))}{\prod_{k=0}^{m(\omega)} |\Jac f(\pi(\sigma^k(\omega)))|^{-T(q)}(e^{g(\pi(\sigma^k(\omega)))})^q} \le C_2'.
\end{equation}

Moreover, we have
\begin{equation}\label{eq_4.12_3}
\prod_{k=0}^{m(\omega)+1} |\Jac f(\pi(\sigma^k(\omega)))|^{-1} \le r \text{ and } \prod_{k=0}^{m(\omega)-1} |\Jac f(\pi(\sigma^k(\omega)))|^{-1} \ge r.
\end{equation}

Combining \eqref{eq_4.12_1}, \eqref{eq_4.12_2} and \eqref{eq_4.12_3}, and for $N$ large enough, we obtain
\begin{align*}
\left(\frac{C_1'}{C_2^q}\right)r^{T(q)}\nu(U(x,\ep,\kappa,m(\omega)))^q \le 
\nu_q(U(x,\ep,\kappa,m(\omega))) \le \left(\frac{C_2'}{C_1^q}\right) r^{T(q)}\nu(U(x,\ep,\kappa,m(\omega)))^q
\end{align*}
which implies
\begin{align}\label{eq_4.12_4}
\frac{\log \left(\frac{C_2'}{C_1^q}\right)}{\log r} + T(q) + \frac{q \log \nu(U(x,\ep,\kappa,m(\omega)))}{\log r} &\le \frac{\log \nu_q(U(x,\ep,\kappa,m(\omega)))}{\log r} \\ \nonumber
&\le \frac{\log \left(\frac{C_1'}{C_2^q}\right)}{\log r} + T(q) + \frac{q \log \nu(U(x,\ep,\kappa,m(\omega)))}{\log r}.
\end{align}

Taking $\lim_{\ep \to 0}\lim_{\kappa \to 0}\lim_{N \to \infty}$ on each term in \eqref{eq_4.12_4}, by Lemma \ref{lem_4.3} and the fact that $x \in J_\nu(\alpha(q))$, we obtain
$\delta_{\nu_q,x}= T(q) + q \delta_{v,x} = T(q) + q\alpha(q)$.

\medskip

Statement (2) follows from the fact that $\nu_q$ is non-atomic, Lemma \ref{lem_3_part1}, \eqref{nu_q_local_dim} and \cite[Proposition 4.22]{BHMane}. This completes the proof of the lemma.
\end{proof}

\begin{corollary}\label{cor_hat_X}
For every $q \in \mathbb R$, we have
\begin{enumerate}
\item $\nu_q(\pi(\hat{X}_q)) =1$;
\item $\delta_{\nu_q,x} = T(q)+q\alpha(q)$ for every $x \in \pi(\hat{X}_q)$; 
\item $\VD_{f,\nu_q}(\pi(\hat{X}_q)) = T(q)+q\alpha(q)$. 
\end{enumerate}
\end{corollary} 
\begin{proof}
Statement (1) is Lemma \ref{lem_2}. Statement (2) and (3) follows from Lemma \ref{lem_3_part1} and Lemma \ref{lem_3_part2}. 
\end{proof}

\subsection{Proof of Theorem \ref{thm_main_1}}
\begin{proof}[Proof of Theorem \ref{thm_main_1}]
By Lemmas \ref{lem_temp_an}, \ref{lem_temp_de}, \ref{lem_temp_conv} and Remark \ref{rmk_T(0)}, to prove statement (1), it remains to show $T(1) = 0$. The argument is near identical to \cite[Theorem 21.1.4 (b)]{Urbanski22}. By definition, $T(1)$ is the unique number such that $P(\phi_{1,T(1)})=0$. Since, by assumption, $P(g) = 0$ and $g = \phi_{1,0}$, we conclude that $T(1)=0$. 

\medskip

For (2), $\hat{S}_\nu(\alpha(q)) = T(q) + q \alpha(q)$ follows from Proposition \ref{prop_lem3} (3) and the definition of $\hat{S}_\nu(\alpha(q))$. 
Using the fact $\mathcal{P}(q,T(q)) = 0$ and Theorem \ref{thm_BHMane_1.1}, we have
$\VD_f(\nu_q) = T(q) +q\alpha(q)$.
Since $T(q)$ is real-analytic and $\alpha(q)=-T'(q)$, we see that $\hat{S}_\nu(\alpha(q))$ is analytic in $q$.

\medskip

For (3), we show that $\alpha \mapsto \hat{S}_\nu(\alpha)$ is concave. Then the conclusion follows from (1) and (2).
Taking derivative with respect to $q$ on both sides of $\hat{S}_\nu(\alpha(q)) = T(q)+q\alpha(q)$ and using $T'(q)=-\alpha(q)$ and $\nu \neq \nu_0$ (so that $\alpha'(q)\neq 0$) we obtain
\begin{equation*}
\frac{d}{d\alpha}\hat{S}_{\nu}(\alpha(q)) = q.
\end{equation*}
Taking derivative with respect to $q$ on both sides of the above equation and recalling from Lemma \ref{lem_temp_conv} that $T$ is strictly convex in this case we obtain
\begin{equation*}
\frac{d^2}{d\alpha^2}\hat{S}_{\nu}(\alpha(q)) = \frac{1}{\alpha'(q)}<0. 
\end{equation*}
Then it is straightforward to see that the function $\alpha \mapsto -\hat{S}_\nu(-\alpha)$ is strictly convex and form a Legendre transform pair with $q \mapsto T(q)$.

For (4), if $\nu = \nu_0$, by the proof of Lemma \ref{lem_temp_conv}, we see that $T''(q) = 0$. Therefore, $T'(q) = -\alpha(q)$ is constant.
\end{proof}

\section{Proof of Theorem \ref{thm_main_1.5}} \label{sec_main_1.5}
By Lemma \ref{lem_3_part1} and Remark \ref{rmk_Jnu_invariant}, we see that $J_\nu(\alpha(q))$ is invariant. Therefore it is an expanding set.
Let $\mu_q$ be the $\delta_{J_{\nu}(\alpha(q))}$-volume conformal measure on $J_\nu(\alpha(q))$ as in \cite[Theorem 5.6]{BHMane}.

Recall from \cite{BHMane} that a probability measure $\nu$ on $J_{\nu}(\alpha(q))$ is {\it $t$-volume-conformal on $J_{\nu}(\alpha(q))$}
if, for every Borel subset $A\subset J_{\nu}(\alpha(q))$ on which $f$ is invertible, we have
\begin{equation*}
	\nu (f (A)) =  \int_A |\Jac f|^t d\nu.
\end{equation*}
We define
\[\begin{aligned}
	\delta_{J_{\nu}(\alpha(q))} (f)  
	& := \inf \left\{ t \ge 0 \colon \mbox{ there exists  a $t$-volume-conformal
		measure on } J_{\nu}(\alpha(q))\right\}. 
\end{aligned}\]

\begin{lemma}\label{lem_comp_lvd}
The $\delta_{J_{\nu}(\alpha(q))}$-volume conformal measure $\mu_q$ supported on $J_{\nu}(\alpha(q))$ and the equilibrium state $\nu_q$ of $\phi_{q}$ are the same; $\mu_q = \nu_q$.
\end{lemma}
\begin{proof}
We first show that $\delta_{J_{\nu}(\alpha(q))} = T(q)+q\alpha(q)$. Since $\mu_q$ is the $\delta_{J_{\nu}(\alpha(q))}$-volume conformal measure on $J_{\nu}(\alpha(q))$, then by \cite[Theorem 5.6]{BHMane}, this $\delta_{J_{\nu}(\alpha(q))}$ is such that $\mathcal{P}_{J_{\nu}(\alpha(q))}^+(\phi_{0,\delta_{J_{\nu}(\alpha(q))}}) = 0$. Then for $\nu_q$, we have
\begin{align*}
    0&=\mathcal{P}_{J_{\nu}(\alpha(q))}^+(\phi_{0,\delta_{J_{\nu}(\alpha(q))}})\geq h_{\nu_q}(f)-t\int_{J_{\nu}(\alpha(q))}\log|\Jac f|d\nu_q
\end{align*}
which implies $\delta_{J_{\nu}(\alpha(q))}\geq \frac{h_{\nu_q}}{L_{\nu_q}} = T(q)+q\alpha(q).$

Now define $\alpha(\mu_q) = \frac{\int_{J_{\nu}(\alpha(q))}g d\mu_q}{-\int_{J_{\nu}(\alpha(q))} \log|\Jac f| d\mu_q}$, by Birkhoffs Ergodic theorem and Lemma \ref{lem_3_part1},
\begin{equation}\label{alphamu}
    \alpha(\mu_q) = \frac{\int_{X}g d\mu_q}{-\int_{X} \log|\Jac f| d\mu_q} =\lim_{n \to \infty} \frac{\sum_{k=0}^{n-1} g(\pi(\sigma^k(\omega)))}{\sum_{k=0}^{n-1} \log  |\Jac f(\pi(\sigma^k(\omega)))|^{-1}} = \alpha(q)
\end{equation}
as the ratio of Birkhoff sums is constant and equal to $\alpha(q)$ on $J_{\nu}(\alpha(q))$ and $\mu_q$ is fully supported on $J_{\nu}(\alpha(q))$.  And so, as $T(q)$ is the unique value for which $\cP(q,T(q)) =0$ then for $\mu_q$, we have 
\begin{align}
    0&=\cP(q,T(q)) \geq h_{\mu_q}(f)-T(q)\int_{J_{\nu}(\alpha(q))}\log|\Jac f|d\mu_q+q\int_{J_{\nu}(\alpha(q))} g d\mu_q \notag \\&=h_{\mu_q}(f)-\delta_{J_{\nu}(\alpha(q))}\int_{J_{\nu}(\alpha(q))}\log|\Jac f|d\mu_q+(\delta_{J_{\nu}(\alpha(q))}-T(q))\int_{J_{\nu}(\alpha(q))}\log|\Jac f|d\mu_q+q\int_{J_{\nu}(\alpha(q))} g d\mu_q \notag
    \\& = (\delta_{J_{\nu}(\alpha(q))}-T(q))\int_{J_{\nu}(\alpha(q))}\log|\Jac f|d\mu_q+q\int_{J_{\nu}(\alpha(q))} g d\mu_q. \label{free_energy_muq} \\ \notag
\end{align}
which implies $\delta_{J_{\nu}(\alpha(q))} \leq T(q)+q\alpha(\mu_q) = T(q)+q\alpha(q).$

Hence $\delta_{J_{\nu}(\alpha(q))} = T(q)+q\alpha(q)$. For this value of $\delta_{J_{\nu}(\alpha(q))}$ it is easy to see from \eqref{free_energy_muq} that the free energy of $\mu_q$ with respect to the potential $\phi_q$ is $0 = \mathcal{P}(q,T(q))$, so $\mu_q$ is the unique equilibrium state for $\phi_q$; $\nu_q = \mu_q$. 
\end{proof}

Now we give a proof of Theorem \ref{thm_main_1.5}.
\begin{proof}[Proof of Theorem \ref{thm_main_1.5}]
By \cite[Theorem 5.6]{BHMane}, Lemma \ref{lem_comp_lvd} and Theorem \ref{thm_main_1} (2), we have $$\VD_f(J_\nu(\alpha(q))) = \VD_f(\mu_q)=\VD_f(\nu_q)= T(q) + q\alpha(q).$$ On the other hand, by Theorem \ref{thm_main_1} (2), we have $\VD_{f,\nu_q}(J_\nu(\alpha(q))) = T(q) + q\alpha(q)$. Therefore $\VD_{f,\nu_q}(J_\nu(\alpha(q))) = \VD_f(J_\nu(\alpha(q)))$, i.e., $S(\alpha(q)) = \hat{S}(\alpha(q))$. The conclusion follows from Theorem \ref{thm_main_1}.
\end{proof}

\section{Proof of Theorem \ref{thm_main_2}} \label{sec_main_2}
In this section, we prove Theorem \ref{thm_main_2}. We first prove some preparatory lemmas in Section \ref{sec_5.1} and prove Theorem \ref{thm_main_2} in Section \ref{sec_5.2}.

\subsection{Preliminary lemmas} \label{sec_5.1}

\begin{lemma}\label{thm_main_var_temp}
For each $q \in \mathbb R$, we have 
$$-T(q) = \inf_{\rho \in \mathcal{M}^+_{X}(f)} \frac{h_{\rho}(f)+q\int_X g d\rho}{-\int_X \log |\Jac f| d\rho}.$$
%where the supremum is taken over all the invariant ergodic probability measures $\rho$ on $X$. 
Moreover, the supremum is realized at $\rho = \nu_q$.
\end{lemma}

\begin{proof}
Since $\mathcal{P}(\phi_g)=0$, we have 
$0 = \sup_{\rho} h_{\rho}+\int_X qg-T(q)\log|\Jac f| d\rho$, where the supremum is taken over invariant ergodic probability measures on $X$ and the supremum is realized at $\nu_q$. Therefore
$$T(q) = \sup_{\rho} \frac{h_\rho + q\int_X gd\rho}{\int_X \log|\Jac f| d\rho}$$
and the conclusion follows.
\end{proof}

\begin{lemma} \label{lem_alpha12}
For any $q \in \mathbb R$, we have
$$\alpha_1 = \inf_{\rho \in \mathcal{M}^+_{X}(f)}\frac{-\int_X g d\rho}{\int_X \log |\Jac f| d\rho} \text{ and } \alpha_2 = \sup_{\rho \in \mathcal{M}^+_{X}(f)}\frac{-\int_X g d\rho}{\int_X \log |\Jac f| d\rho}.$$

\end{lemma}
\begin{proof}
We prove the statement for $\alpha_1$. The statement for $\alpha_2$ can be proved using the same argument.
We first observe that
$\alpha_1 \ge \inf_{\mu \in \mathcal{M}^+_{X}(f)}\frac{-\int_X g d\mu}{\int_X \log |\Jac f| d\mu}$ as by definition, $\alpha_1 := \inf_{q\in \mathbb R}\frac{-\int_X g d\nu_q}{\int_X \log |\Jac f| d\nu_q}$.

Now we prove $\alpha_1 \le \inf_{\mu \in \mathcal{M}_{X}(f)}\frac{-\int_X g d\mu}{\int_X \log |\Jac f| d\mu}$. Since $\alpha_1 \le \alpha(q)$ for any $q \in \mathbb{R}$, by Proposition \ref{lem_2} (3), Theorem \ref{thm_main_var_temp} and $\VD_f(\nu)\le 2$ for any $\nu \in \mathcal{M}_X^+(f)$ (see \cite[Lemma 4.18]{BHMane}), for any $q \ge 0$, we have
\begin{align*}
-\VD_{f,\nu_q}(\pi(\hat{X}_q)) + q\alpha_1 & \le -\VD_{f,\nu_q}(\pi(\hat{X}_q)) +q\alpha(q)= -T(q) \\
&=\inf_{\rho} \left\{\VD_f(\rho)+q \frac{\int_X g d\rho}{-\int_X \log |\Jac f|d\rho} \right\}\\
&\le 2 + q \inf_{\rho} \frac{\int_X g d\rho}{-\int_X \log |\Jac f|d\rho}.
\end{align*}
Since $\VD_{f,\nu_q}(\pi(\hat{X}_q))\le 2$ for any $q \in \mathbb R$ (see \cite[Lemma 4.10]{BHMane}), we have for any $q \ge 0$, $$q\alpha_1 \le 4 + q \inf_{\rho} \frac{\int_X g d\rho}{-\int_X \log |\Jac f|d\rho}.$$
Dividing both sides by $q$ and letting $q \to \infty$, we obtain
$$\alpha_1 \le \inf_{\rho} \frac{-\int_X g d\rho}{\int_X \log |\Jac f|d\rho}.$$
This completes the proof.
\end{proof}

\subsection{Proof of Theorem \ref{thm_main_2}} \label{sec_5.2}
\begin{proof}[Proof of Theorem \ref{thm_main_2}]
We first show that $\alpha_1 = \inf_{x\in X} \underline{\delta}_{\nu,x}$. Suppose $\alpha_1 < \inf_{x\in X} \underline{\delta}_{\nu,x}$. This implies that there exists a large $q_0 \gg 1$ such that $J_{\nu}(\alpha(q_0))$ is empty, which contradicts Proposition \ref{prop_lem3}. Therefore $\alpha_1 \ge \inf_{x\in X} \underline{\delta}_{\nu,x}$. 

Now suppose $\alpha_1 > \inf_{x\in X} \underline{\delta}_{\nu,x}$. Then using Lemma \ref{lem_3_part1}, there exist $\eta>0, x\in X$ and a sequence $\{n_k\} \subset \mathbb N$ such that for each $k \ge 1$, we have
$$\frac{-\sum_{j=0}^{n_k}g(f^j(x))}{\sum_{j=0}^{n_k}\log|\Jac f(f^j(x))|} \le \alpha_1 -\eta.$$

For each $k\ge 1$, consider the measure $\rho_k := \frac{1}{n_k}\sum_{j=0}^{n_k} \delta_{f^j(x)}$ where $\delta_y$ denotes the Dirac mass at the point $y$. Then we have
$$\frac{-\int_X gd\rho_k}{\int_X \log |\Jac f| d\rho_k} \le \alpha_1 -\eta.$$
Let $\rho$ be an accumulation point of the sequence $\{\rho_k\}_{k \ge 1}$ of measures. Then taking the limit $k \to \infty$, we have
$$\frac{-\int_X gd\rho}{\int_X \log |\Jac f| d\rho} \le \alpha_1 -\eta.$$
If $\rho$ is ergodic, then this contradicts Lemma \ref{lem_alpha12}. If $\rho$ is not ergodic, then since the set of periodic measures is dense, we approximate $\rho$ by a sequence of periodic measures $\tilde{\rho}_n$. Then there exists a sufficiently large $n$ such that $\frac{-\int_X gd\tilde{\rho}_n}{\int_X \log |\Jac f| d\tilde{\rho}_n} \le \alpha_1 -\eta/2$, which also contradicts Lemma \ref{lem_alpha12}. Therefore we must have $\alpha_1 = \inf_{x\in X} \underline{\delta}_{\nu,x}$.  

The statement $\alpha_2 = \sup_{x\in X} \overline{\delta}_{\nu,x}$ can be proved in a similar way.
\end{proof}

\section{Proof of Theorem \ref{thm_main_3}} \label{sec_main_3}
\begin{proof}[Proof of Theorem \ref{thm_main_3}]

By Corollary \ref{cor_hat_X} and the proof of Theorem \ref{thm_main_1}, we see that Theorem \ref{thm_main_1} (2) and (3) hold for $\hat{\mathcal{F}}_\nu(\alpha(q))$.
For ${\mathcal{F}}_\nu(\alpha(q))$, by Lemmas \ref{lem_3_part1} and \ref{lem_comp_lvd}, we have $\hat{\mathcal{F}}_\nu(\alpha(q))= T(q) + q\alpha(q) = \mathcal{F}_\nu(\alpha(q))$. This completes the proof.
\end{proof}

\section{Proof of Corollary \ref{thm_main_4}} \label{sec_main_4}
\begin{proof}[Proof of Corollary \ref{thm_main_4}]
Since for every $\alpha \in [\alpha_1,\alpha_2]$, we have $\hat{S}_{\nu}(\al) = S_{\nu}(\al) = \hat{\mathcal{F}}_{\nu}(\al) = \mathcal{F}_{\nu}(\al)$, we prove the corollary for $\hat{S}_{\nu}(\al)$. By Theorem \ref{thm_main_1}, Lemma \ref{lem_comp_lvd} and \cite[Theorem 5.6]{BHMane}
we have
\begin{align*}
\hat{S}_{\nu}(\al(q)) = \VD_f(\nu_q) = \VD_f(\mu_q) =\VD_f(J_\nu(\alpha(q))).
\end{align*}
For any other measure $\rho \in \mathcal{M}_{J_{\nu}(\alpha(q))}^+(f)$, since $\rho$ is ergodic, by \eqref{alphamu} we have $\alpha(\rho) = \alpha(q)$. Thus, by the variational principal, 
\begin{align*}
0 &= \cP(q,T(q)) \geq h_{\rho}(f) + q\int g d\rho -T(q)\int \log|\Jac f|d\rho\\
&\implies 0\geq \VD_f(\rho)+ q(-\alpha(q))-T(q)\\
&\implies \VD_f(\nu) \geq \VD_f(\rho).
\end{align*}
This completes the proof.
\end{proof}

\section{Properties of the multifractal decomposition}\label{sec_multi_decomp}
Recall that we have the following multifractal decomposition of the set $X$, i.e.,
$X = J'_\nu \cup \left(\bigcup_{q \in \mathbb{R}} J_\nu(\alpha(q)) \right).$
\begin{lemma}
For every $q \in \mathbb{R}$, the set $J_\nu(\alpha(q))$ is dense in $X$.
\end{lemma}
\begin{proof}
By Lemma \ref{lem_3_part1}, it suffices to show that $\pi(\hat{X}_q)$ is dense in $X$. For any $\om\in \Sigma^{+}_A$, set 
$$
    \alpha(\om) : = \lim_{n\to\infty}\frac{\sum_{k=0}^{n-1} g(\pi(\sigma^k(\omega)))}{\sum_{k=0}^{n-1}  \log|\Jac f(\pi(\sigma^k(\omega)))|^{-1}}.
$$
Now, let $\cO$ be an open set in $X$. Since the Markov partition weakly partitions $X$, there exists an $n$ for which there is an $S\in \bigvee_{i=0}^{n-1}\cR$ with $S\subseteq \cO$. Let $w(S)$ be the finite admissible word encoding $S$. 

Next let $\om\in \hat{X}_q$. Since $f \colon X\to X$ is exact, it is transitive, and thus the associated adjacency matrix $A$ for $\Sigma_A^{+}$ defining the admissibility condition is irreducible, so there exists a finite word $u$ such that $w(S)_{|w(S)|}u\om_0$ is admissible. Then,
$$
    \alpha(\om) = \alpha(w(S)u\om) = \alpha(q).
$$
Hence $\pi(w(S)u\om)\in S\cap\pi(\hat{X}_q)\subseteq \cO$. As $\cO$ was arbitrary, $\pi(\hat{X}_q)$ is dense in $X$. 
\end{proof}

\begin{lemma}
The set $J_\nu'$ is non-empty. 
\end{lemma}
\begin{proof}
We show that there exists $x \in X$ such that $\displaystyle\lim_{\ep \to 0}\lim_{\kappa \to 0}\lim_{N \to \infty} \frac{\log \nu(U(x,\ep,\kappa,N))}{\log {\rm Vol}(U(x,\ep,\kappa,N))^{1/2}}$ does not exist. By Gibbs property \eqref{eq_gibbs_nu}, $\log C_2 + S_Ng(x) \le \log \nu(U(x,\ep,\kappa,N)) \le \log C_1 + S_Ng(x)$. It suffices to show that there exists $x \in X$ such that the limit of Birkhoff averages $\lim_{N \to \infty}\frac{1}{N}S_Ng(x)$ does not exists.

Let $\rho_1,\rho_2$ be two equilibrium states such that $\int_X g d\rho_1 \neq \int_X g d\rho_2$. By Birkhoff’s ergodic theorem, there exist points $y_i$ typical for $\rho_i$ such that 
$\frac{1}{N}S_Ng(y_i) \to \int_X g d\rho_i$. In the mixing subshift of finite type $\Sigma_A^+$, we construct an infinite word $\omega$ by concatenating alternatingly long blocks cut from the orbits of $y_1$ and $y_2$, in such a way that the lengths of the blocks satisfy $L_{1} \ll M_{1} \ll L_2 \ll M_2 \cdots$. Here $L_i$ (resp. $M_i$) is the length of the $i$th block of orbits of $y_1$ (resp. $y_2$).
 
Then the partial sums $\frac{1}{m} S_m g(\pi(\omega))$ will alternating between $\int_X gd\rho_1$ and $\int_X gd\rho_2$. Therefore the limit $\displaystyle \lim_{m \to \infty}\frac{1}{m} S_m g(\pi(\omega))$ does not exist. Thus $\pi(\omega) \in J_\nu'$.
\end{proof}

For conformal dynamical systems, the Hausdorff dimension and entropy of irregular sets have been studied by Barreira-Schmeling \cite{BarSch00}; see also \cite[Chapter 8]{Barreirabook08}.

\bibliographystyle{plain}
\bibliography{VDM.bib}
 
\end{document}